\documentclass[10pt]{article}

\usepackage[utf8]{inputenc}

\usepackage{url}

\usepackage[hidelinks,breaklinks]{hyperref}

\usepackage{geometry}
\usepackage[english]{babel}
\usepackage[backend=biber,style=alphabetic]{biblatex}
\usepackage{diagbox}
\usepackage{mathtools}
\usepackage{amsmath}
\usepackage{amssymb}
\usepackage{mathrsfs}
\usepackage{dsfont}
\usepackage{graphicx}
\usepackage{booktabs}
\usepackage{array}
\usepackage{paralist}
\usepackage{verbatim}
\usepackage{subfig}
\usepackage{tikz}
\usepackage{tikz-cd}
\usetikzlibrary{arrows}
\tikzset{node distance=1.8cm, auto}
\usepackage{csquotes}

\usepackage{lmodern}
\usepackage{sectsty}
\usepackage{amsthm}
\theoremstyle{definition}
\allsectionsfont{\sffamily\mdseries\upshape} 
\newtheorem{theorem}{Theorem}[section]
\newtheorem{proposition}[theorem]{Proposition}
\newtheorem{corollary}[theorem]{Corollary}
\newtheorem{definition}[theorem]{Definition}

\newtheorem{example}[theorem]{Example}

\newtheorem{lemma}[theorem]{Lemma}
\newtheorem{remark}[theorem]{Remark}

\newtheorem{question}[theorem]{Question}

\usepackage{rotating}
\usepackage{graphicx}

\let\phi\varphi
\let\theta\vartheta
\DeclareMathOperator{\Norm}{\operatorname{N}}
\DeclareMathOperator{\Iso}{Iso}

\DeclareMathOperator\Aut{Aut}

\DeclareMathOperator\Tr{Tr}

\DeclareMathOperator\Sym{Sym}

\renewcommand{\:}{\colon}

\renewcommand{\leq}{\leqslant}
\renewcommand{\geq}{\geqslant}

\newcommand{\T}{\mathbf{T}}
\newcommand{\U}{\mathbf{U}}
\renewcommand{\S}{\operatorname{S}}

\newcommand{\bF}{\mathbb{F}}

\newcommand{\F}{\mathbb{F}}
\newcommand{\Z}{\mathbb{Z}}

\renewcommand{\k}{\Bbbk}

\newcommand{\p}{\mathfrak{p}}
\renewcommand{\l}{\mathfrak{l}}

\title{\textsc{The strong Ramanujan bound \\ for Drinfeld modular forms}}

\author{Sjoerd de Vries}
\date{}

\begin{document}
\maketitle
\begin{abstract}
    \noindent

    The Ramanujan bound states that the degree of the trace of the $\mathfrak{p}$-Hecke operator on a space of weight~$k$ Drinfeld cusp forms is at most $d(k-2)/2$, where $\mathfrak{p}$ is a prime of degree~$d$ in $\bF_q[T]$. We prove the strong Ramanujan bound, which states that at level~1, this bound can be improved to $d(k-(q+1))/2$. In addition, we classify when the strong Ramanujan bound is sharp; surprisingly, this turns out to depend on $\mathfrak{p}$ only through the parity of its degree. We also prove results concerning the smallest slope of the Atkin operator and mod~$p$ distributions of Frobenius polynomials of Drinfeld modules over finite fields.
\end{abstract}


\tableofcontents

\section{Introduction}
Drinfeld modular forms of rank~2 are the function field analogues of elliptic modular forms, and can be described as rigid-analytic functions on the Drinfeld upper half-plane with values in a complete algebraically closed field of characteristic~$p$. The theory shares features with that of elliptic modular forms, but is also strikingly different in various aspects.

In previous work of the author, a Ramanujan bound on traces of Hecke operators on Drinfeld modular forms was established. As in the elliptic case, this general (level-agnostic) Ramanujan bound is sharp (see Remark~\ref{rmk:ram-sharp}). However, computations in the case of level $\mathrm{GL}_2(\bF_q[T])$ suggested an improved bound, the \emph{strong Ramanujan bound}, which does not seem to have a classical counterpart. The author previously proved this stronger bound when the Hecke operator is associated to an ideal of degree at most~3, see \cite[Sec.~4]{devries_traces}.

The main objective of the present article is to present a general proof of the strong Ramanujan bound, conjectured in \cite{devries_traces}. At its core, the proof is a statement about the distribution of Frobenius polynomials of rank~2 Drinfeld modules over finite fields, and we prove some additional results in this direction that are of independent interest. In particular, since the size of an isogeny class of Drinfeld modules is described by a class number, the strong Ramanujan bound gives an infinitude of class number relations mod~$p$.




Our main result is as follows. For precise definitions, see Section~\ref{sec:notation}.

\begin{theorem}\label{thm:main_SRB}
    The strong Ramanujan bound holds for all $q$ and all~$\p^n$. In other words, for any prime power $q$, prime ideal $\mathfrak p \trianglelefteq \bF_q[T]$ and $n \geq 1$, we have
    \[
        \deg \Tr(\T^n_{\p} \, | \, \S_{k,l}) \leq  \frac{nd(k-(q+1))}{2} \qquad \text{for all }k,\,l \in \mathbb Z.
    \]
\end{theorem}

It is natural to ask whether the strong Ramanujan bound is the best one can do. In Section~\ref{sec:sharpness}, we show that it is: for every $q$ and $\p^n$, it is sharp infinitely often. Since Hecke operators for different primes~$\p$ tend to behave quite differently in general, the following result came as a surprise.

\begin{theorem}
    For fixed $k$ and $l$, whether the strong Ramanujan bound is sharp on $\S_{k,l}$ depends only on the parity of $\deg \p^n$.
    Furthermore, for every $\T_{\p^n}$, the strong Ramanujan bound is sharp in a set of weights with positive density.
\end{theorem}

We can in fact precisely classify the weights in which the bound is sharp; see Theorem~\ref{thm:sharpness-criterion} for the exact statement. The proof is a computation of the top degree coefficient of the trace, and uncovers a connection between Drinfeld modular forms and closed walks in Kneser graphs (see Remark~\ref{rmk:kneser}).

When $\p$ is a prime of degree~1, our formulas take a particularly nice shape, allowing us to also compute a dual bound through which we can investigate the smallest $\p$-adic slope of the Hecke operator. 

The ideas going into these proofs are applicable more broadly to the study of the distribution mod~$p$ of isogeny classes of Drinfeld modules, which we study in Section~\ref{sec:distributions}, complementing earlier results of Gekeler \cite{gek_frobdist}. The main results in this direction are Theorems~\ref{thm:bottom-uniformity} and~\ref{thm:top-invisibility}.
They can be summarized as follows: 

\begin{theorem}
    The number of isomorphism classes of rank 2 Drinfeld modules over $\bF_{\p^n}$ whose Frobenius polynomial $X^2-aX+b\wp^n$ has $a_0=\alpha$ and $b=\beta$ is $\equiv 1 \pmod p$. The number of those whose top $j$ coefficients $a_N,\dots,a_{N-j+1}$ and $b$ are prescribed is divisible by $q$ provided $j^2<nd$; for $nd$ odd, it is divisible by an explicit higher power $q^{\mu}$. More general sets of prescribed coefficients are also allowed.
\end{theorem}

\subsection{Notation and conventions}\label{sec:notation}
We fix a prime power $q$ and set $A = \bF_q[T]$, with fraction field $K = \bF_q(T)$ of characteristic~$p$. Throughout, $\wp \in A$ will denote a monic irreducible polynomial of degree~$d$ and $n \geq 1$ will be an integer. We write $\p = (\wp)$, $\bF_{\p} := A/\p$, and $\bF_{\p^n}$ for the degree~$n$ field extension of $\bF_{\p}$ (so that $|\bF_{\p^n}| = q^{nd}$). We write $0^0 = 1$, $\deg 0 = - \infty$, and denote by $[X^D]f$ the coefficient of~$X^D$ in a polynomial~$f = f(X)$.

\subsection{Background}

\subsubsection*{Drinfeld modules}

Let $\k$ be a finite field together with an $\bF_q$-algebra homomorphism $\gamma\: A \to \k$.  If $\k = \bF_{\p^n}$, we will assume that $\gamma$ is of the form $A \twoheadrightarrow \bF_{\p} \hookrightarrow \bF_{\p^n}$. Write $\theta := \gamma(T) \in \k$. A Drinfeld module of rank~2 over~$\k$ is then an $\bF_q$-algebra homomorphism
\[
\phi: A \longrightarrow \k\{\tau\}, \qquad T \longmapsto \phi_T := \phi(T) = \theta + g\tau + \Delta\tau^2,
\]
with $\Delta \neq 0$. For fixed $(g,\Delta) \in \k \times \k^\times$, we define\footnote{Note that $\k$ is not determined by $g$ and $\Delta$, so it may be ambiguous over which field $\phi^{(g,\Delta)}$ is considered. In this paper, the base field is always of the form $\bF_{\p^n} \cong \bF_{q^m}$, so this should not cause confusion.} the Drinfeld module $\phi^{(g,\Delta)}$ through the formula $\phi_T = \theta + g\tau + \Delta\tau^2$. If $\phi$ is a Drinfeld module over~$\k$, we denote its $\k$-isomorphism class by $[\phi]$.

Each $\phi$ as above has an associated \emph{Frobenius polynomial} $c_\phi(X) \in A[X]$. If $\k = \bF_{\p^n}$, then $c_\phi(X) = X^2 - a_\phi X + b_\phi \wp^n$ for some $a_\phi \in A$ and $b_\phi \in \bF_q^\times$. If $\phi = \phi^{(g,\Delta)}$, then we also write $a(g,\Delta) := a_\phi$, $b(g,\Delta) := b_\phi$. Moreover, we write $a = \sum_{i=0}^{\deg(a)} a_iT^i$, $a_\phi = \sum_{i=0}^{\deg(a_\phi)} a_{i,\phi}T^i$, and
\[
a(g,\Delta) = \sum_{i \geq 0} a_i(g,\Delta)T^i.
\]
Given $a \in A$ and $b \in \bF_q^\times$, we denote by $\Iso_{\p^n}(a,b)$ the \emph{isogeny class} associated to $(a,b)$, defined as the set of isomorphism classes of Drinfeld modules $\phi$ with $c_\phi(X) = X^2 - aX + b\wp^n$.

\begin{remark}\label{rmk:class-numbers}
    The numbers $\# \Iso_{\p^n}(a,b)$ are given by class numbers of orders in quadratic extensions of~$K$ or quaternion algebras over~$K$. For instance, in the ordinary case, $\# \Iso_{\p^n}(a,b)$ is the Hurwitz class number of the order $A[\pi]$ in the imaginary quadratic extension $K(\pi) = K[X]/(X^2-aX+b\wp^n)$, cf.~\cite[Appendix~A]{devries_traces}.
\end{remark}

\subsubsection*{Drinfeld modular forms}
Given integers $k$ and $l$, we write $\S_{k,l}$ for the space of Drinfeld cusp forms of level $\mathrm{GL}_2(\bF_q[T])$, weight $k$ and type $l$. It comes equipped with a Hecke action; the Hecke operator associated to $\p$ is denoted $\T_{\p}$, weight and type being implicit.

The Hecke trace formula then states that
\[
\Tr(\T_{\p}^n \, | \, \S_{k+2,l}) = \sum_{j = 0}^{\lceil k/2 \rceil - 1} (-1)^j \binom{k-j}{j} \sum_{a,b} \# \Iso_{\p^n}(a,b) a^{k-2j}b^{j+l-k-1}\wp^{nj}.
\]
The strong Ramanujan bound is defined and studied in \cite[Sec.~4]{devries_traces}. Given $\mathfrak p^n$, it states that
\[
\deg \Tr(\T^n_{\p} \, | \, \S_{k,l}) \leq  \frac{nd(k-(q+1))}{2} \qquad \text{for all }k,\,l.
\]
We will also consider the following related quantities: for integers $k \geq 0$ and $t \in \Z$, write
\[
W_{k,t} := \sum_{a,b} \#\Iso_{\p^n}(a,b)a^kb^t \in \bF_q[T],
\]
and note that $W_{k,t} \neq 0$ only if $k \equiv -2t \pmod{q-1}$ by \cite[Rmk.~2.22]{devries_traces}.
For $N = \lfloor nd/2 \rfloor$ and $\mathbf{v} \in \Z_{\geq 0}^{N+1}$, write
\[
\mu(\mathbf{v},t) = \sum_{a,b} \#\Iso_{\p^n}(a,b) a_0^{v_0}a_1^{v_1}\cdots a_N^{v_N}b^t \in \bF_q.
\]
We also write $|\mathbf{v}| := \sum v_i$ and $||\mathbf{v}|| := \sum i v_i$. Note that $W_{k,t}$ and $\mu(\mathbf{v},t)$ can be viewed as moments of the distribution defined by $\Iso_{\p^n}$.

Since traces of Hecke operators are explicit $\Z[\wp^n]$-linear combinations of the $W_{k,t}$, the strong Ramanujan bound is equivalent to the statement that
\[
\deg W_{k,t} \leq \frac{nd(k-(q-1))}{2} \qquad \text{for all } k,\,t.
\]

The relationship with the $\mu(\mathbf{v},t)$ is as follows:
\begin{equation}\label{eq:W-mu}
W_{k,t}\;=\;\sum_{m=0}^{kN}T^{m}\!\!\sum_{\substack{\mathbf{v}\,:\,|\mathbf{v}|=k\\ \|\mathbf{v}\|=m}}\!\!
\binom{k}{v_0,v_1,\dots,v_N}\,\mu(\mathbf{v},t).
\end{equation}


This yields the following sufficient condition for the strong Ramanujan bound to hold.

\begin{proposition}[{\cite[Prop.~4.7]{devries_traces}}]\label{prop:sufficient}
    Suppose $\mu(\mathbf{v},t) = 0$ 
    whenever $2 ||\mathbf{v}|| > nd(|\mathbf{v}| -(q-1))$. Then the strong Ramanujan bound holds for $\p^n$.
\end{proposition}

We will prove the strong Ramanujan bound by showing that the condition in Proposition~\ref{prop:sufficient} is always satisfied. 

\subsection{Acknowledgments}
I am grateful to Bjorn Poonen for his suggestion to look into the Chevalley--Warning theorem and its generalizations, which improved Theorem~\ref{thm:top-invisibility} from a mod~$p$ congruence to a $p$-power congruence.

\subsubsection*{Generative AI statement}
The author acknowledges discussions with the AI models Claude Opus~4.8 and Claude Fable~5 in the preparation of this article, especially for developing proof strategies and proofreading. None of the text in this article is written by AI and the author takes sole responsibility for the contents.

\section{Proving the strong Ramanujan bound}\label{sec:srb-proof}
In this section, we will deduce Theorem~\ref{thm:main_SRB} from the following result.

\begin{theorem}\label{thm:SRB_upgrade}
    Let $N = \lfloor nd/2 \rfloor$. Let $\mathbf{v} =(v_0,\ldots,v_N)$ be a tuple of nonnegative integers and let $t \in \mathbb Z$. Assume that $2 ||\mathbf{v}|| > nd \left( |\mathbf{v}| - q + 1 \right)$. Then we have
    \[
     \mu(\mathbf{v},t) = \sum_{a \in A} \sum_{b \in \bF_{q}^\times} \#\text{Iso}_{\mathfrak p^n}(a,b) a_0^{v_0}a_1^{v_1}\cdots a_N^{v_N} b^t = 0.
    \]
\end{theorem}

Together with Proposition~\ref{prop:sufficient}, this implies Theorem~\ref{thm:main_SRB}.

Since the theorem holds for any~$t$, it also holds for each~$b$ individually, as stated in the following corollary. The proof is an easier version of that of Lemma~\ref{lem:determine}.

\begin{corollary}\label{cor:fixed-b}
    Fix $b \in \bF_q^\times$ and let $\mathbf{v}$ satisfy the hypothesis from Theorem~\ref{thm:SRB_upgrade}. Then 
    \[
    \sum_{a \in A}\#\text{Iso}_{\p^n}(a,b)a_0^{v_0}a_1^{v_1}\cdots a_N^{v_N} = 0.
    \]
\end{corollary}

Theorem~\ref{thm:SRB_upgrade} and Corollary~\ref{cor:fixed-b} are already interesting in their own right, as they establish a large family of class number relations that are otherwise non-obvious, cf.~Remark~\ref{rmk:class-numbers} and Example~\ref{ex:class-number-rels}.

The proof of Theorem~\ref{thm:SRB_upgrade} proceeds as follows. First, we observe that one can ignore isogeny classes entirely and instead work with the full set of Drinfeld modules, thus obtaining a sum over a different parameter space: $\k \times \k^\times$ instead of $A \times \bF_q^\times$. After rewriting the formulas in terms of the new parameters, we can prove the vanishing of the required terms by elementary finite field arithmetic and a digit-sum lemma.

It is a straightforward fact that $[\phi^{(g,\Delta)}] = [\phi^{(g',\Delta')}]$ if and only if there exists $\lambda \in \k^\times$ such that $g = \lambda^{q-1}g'$ and $\Delta = \lambda^{q^2-1}\Delta'$. In particular, $\Aut(\phi) \cong \bF_q^\times$ if $g \neq 0$ and $\Aut(\phi) \cong \bF_{q^{(2,m)}}^\times$ if $g = 0$ and $\k \cong \bF_{q^m}$. Using this, we show the following easy but important lemma.

\begin{lemma}\label{lem:domain}
    Let $\k$ be a field extension of~$\bF_q$ of degree~$m$. Let $f$ be a function on isomorphism classes of Drinfeld modules over $\k$ taking values in an $\bF_p$-vector space $V$. Then 
    \[
    \sum_{[\phi]} f[\phi] = \sum_{(g,\Delta) \in \k \times \k^\times} f[\phi^{(g,\Delta)}].
    \]
\end{lemma}

\begin{proof}
    By definition, each Drinfeld module $\phi$ is of the form $\phi^{(g,\Delta)}$ for some $(g,\Delta) \in \k \times \k^\times$. Consider the $\k^\times$-action on such pairs given by $\lambda \star (g,\Delta) = (\lambda^{q-1}g,\lambda^{q^2-1}\Delta)$. Then $[\phi^{(g,\Delta)}] = [\phi^{(g',\Delta')}]$ if and only if $(g,\Delta)$ lies in the $\star$-orbit of $(g',\Delta')$. Therefore
    \[
    \sum_{(g,\Delta) \in \k \times \k^\times} f[\phi^{(g,\Delta)}] = \sum_{[\phi]} \# \text{Orb}_\star(\phi) f[\phi].
    \]
    By the Orbit-Stabilizer Theorem, we have 
    \[
    \# \text{Orb}_\star(\phi) = \frac{\#\k^\times}{\# \text{Stab}_\star(\phi)} = \frac{q^m-1}{\# \Aut(\phi)}.
    \]
    If $\#\Aut(\phi) = q-1$, then $\#\text{Orb}_\star(\phi) = 1 + q + \ldots + q^{m-1} \equiv_p 1$. If $\#\Aut(\phi) = q^2-1$, then $m$ is even and again $\#\text{Orb}_\star(\phi) = 1 + q^2 + \ldots + q^{m -2} \equiv_p 1$. This finishes the proof.
\end{proof}

\begin{corollary}\label{cor:mass}
Let $\k = \bF_{\p^n}$. For all $\mathbf{v} \in \mathbb Z_{\geq 0}^{N+1}$ and $t \in \mathbb Z$,
\[
 \mu(\mathbf{v},t) = \sum_{a,b} \#\Iso_{\p^n}(a,b)\, \prod_{i=0}^N a_i^{v_i} b^t
= \sum_{(g,\Delta) \in \k \times \k^\times} b(g,\Delta)^t \prod_{i=0}^N a_i(g,\Delta)^{v_i}
\quad \text{in } \F_q.
\]
\end{corollary}

\begin{proof}
Let $f$ be any function on $A \times \bF_{q}^\times$ with values in an abelian group. By definition of $\Iso_{\p^n}(a,b)$ (see \cite[Def.~2.21]{devries_traces}), we have
\[
\sum_{a,b} \#\Iso_{\p^n}(a,b)f(a,b) = \sum_{[\phi]} f(a_\phi,b_\phi).
\]
Now apply
Lemma~\ref{lem:domain} with $V = \F_q$ and $f[\phi] = f(a_\phi,b_\phi) = b_\phi^t\prod_{i=0}^N a_{i,\phi}^{v_i} $.
\end{proof}

The reason why having a sum over $\k \times \k^\times$ may prove useful is the following standard fact, which we state here for completeness.

\begin{lemma}\label{lem:orth}
Let $\k \cong \bF_{q^m}$. Then for any $E\in \Z_{\geq 0}$ and $C \in \mathbb Z$,
\[
\sum_{g\in \k}g^E=\begin{cases}-1 & \text{if } E>0\text{ and }(q^m-1)\mid E;\\0 &\text{otherwise},\end{cases}
\qquad
\sum_{\Delta\in \k^\times}\Delta^C=\begin{cases}-1 & \text{if }(q^m-1)\mid C;\\0 &\text{otherwise}.\end{cases}
\]
\end{lemma}

Our next step is to find formulas for $a(g,\Delta)$ and $b(g,\Delta)$. We will essentially follow the method of \cite{caruso-leudiere}, specialized to our setting.

Fix $(g,\Delta) \in \k \times \k^\times$ and write $\varphi = \varphi^{(g,\Delta)}$. Consider the associated Anderson motive, i.e., the $\k$-vector space $M(\varphi) := \k\{\tau\}$, viewed as
a module over $\k \otimes_{\F_q} A = \k[T]$ by
\[
(c \otimes a)\cdot m := c\, m\, \varphi_a,
\]
endowed with the operator $\tau_M(m) := \tau m$. Note that $\tau_M$ is $(\sigma \otimes \text{id})$-semilinear: explicitly, $\tau_M((c \otimes a)m) = (c^q \otimes a)\tau_M(m)$.

\begin{lemma}\label{lem:motive-basis} Let $\k = \bF_{\p^n}$ and $\phi = \phi^{(g,\Delta)}$.
\begin{enumerate}
    \item $M(\phi)$ is free of rank~$2$ over $\k[T]$ with basis $\{1, \tau\}$;
    \item With respect to the above basis, we have
\[
\tau_M \begin{pmatrix} z \\ w \end{pmatrix}
= F \cdot \sigma\!\begin{pmatrix} z \\ w \end{pmatrix},
\qquad
F := \begin{pmatrix} 0 & \Delta^{-1}(T - \theta) \\ 1 & -\Delta^{-1} g\end{pmatrix} \in
\mathrm{M}_2(\k[T]);
\]
\item $\pi_\phi = \tau_M^{nd}$ is a $\k[T]$-linear map acting on $M(\phi)$ with matrix
\[
\Phi := F\,\sigma(F)\,\sigma^2(F)\cdots \sigma^{nd-1}(F);
\]
\item Set $a := \Tr(\Phi)$ and $D := \det(\Phi)$.
Then $a, D \in \bF_q[T]$;
\item $a = a(g,\Delta)$ and $D = b(g,\Delta)\wp^n$ with $b(g,\Delta) = (-1)^{nd}\, \Norm_{\k/\bF_q}(\Delta)^{-1} \in \bF_q^\times$;
\item $c_\phi(X) = X^2 - aX + D$.
\end{enumerate} 

\end{lemma}

\begin{proof}
1. This is standard, see e.g. \cite[Lem.~3.4.4]{papikian}. 

2. Since $T \cdot 1 = \varphi_T = \theta + g\tau + \Delta\tau^2$, we have
\begin{equation}\label{eq:tausq}
\tau^2 = \Delta^{-1}(T-\theta)\cdot 1 - \Delta^{-1}g \cdot \tau,
\end{equation}
which yields the expression for~$F$. 

3. Iterating the equality $\tau_M = F\sigma$, we find that $\tau_M^i = F \sigma(F) \cdots \sigma^{i-1}(F)\sigma^i$ for all $i \geq 1$. Since $\sigma^{nd} = \text{id}$, we get the claimed expression for $\pi$.

4. A priori, we have $a(g,\Delta), D \in \k[T]$. Since $\sigma^{nd}(F) = F$, we have
$\sigma(\Phi) = \sigma(F)\cdots\sigma^{nd-1}(F)\,F = F^{-1}\,\Phi\, F$ over $\k(T)$.
Hence $\Tr(\Phi)$ and $\det(\Phi)$ are
$\sigma$-invariant, i.e., lie in $\F_q[T]$.

5. It is a standard fact that $c_\phi(X)$ is the characteristic polynomial of~$\pi_{\phi}$ acting on~$M(\phi)$, which gives $a = a(g,\Delta)$. Moreover, $b(g,\Delta) \wp^n = \det \Phi$, which can be computed as 
\[
\det \Phi = \prod_{i=0}^{nd-1}\sigma^i(\det F) = \prod_{i=0}^{nd-1}
\bigl(-\Delta^{-q^i}(T-\theta^{q^i})\bigr) = (-1)^{nd} \Norm_{\k/\F_q}(\Delta)^{-1}
\prod_{i=0}^{nd-1}(T - \theta^{q^i}).
\]
The $d$ roots of $\wp(T)$ over $\k$ are $\theta, \theta^q, \dots, \theta^{q^{d-1}}$, and the list $i = 0, \dots, nd-1$ runs through them exactly $n$ times each; hence the product is $\wp(T)^n$.

6. This follows from points 4 and 5.
\end{proof}

\begin{remark}
    Since $b(g,\Delta)$ only depends on~$\Delta$ by Lemma~\ref{lem:motive-basis}.5, we will simply denote it by~$b(\Delta)$. 
\end{remark}

\subsubsection*{Disconnected subsets}

For $i \in \Z/m$, set
\[
u_i := \Delta^{-q^i}\,(T - \theta^{q^i}) \in \k[T], \qquad
w_i := -\Delta^{-q^i} g^{q^i} \in \k,
\]
so that $\sigma^i(F) = \begin{pmatrix} 0 & u_i \\ 1 & w_i \end{pmatrix}$. Let $\Phi$ be the matrix of $\pi_\phi$ as in Lemma~\ref{lem:motive-basis}. We will express $a(g,\Delta) = \Tr(\Phi)$ in terms of the $u_i$ and~$w_i$. We first set up some notation.

\begin{definition}
    Call a subset $S \subseteq \Z/m$ \emph{disconnected} if $i \in S \implies i + 1 \notin S$.
For $S$ disconnected, put $\mathrm{Cov}(S) := \bigcup_{i \in S}\{i-1, i\}$ and $U(S) := (\Z/m) \setminus
\mathrm{Cov}(S)$ ($U$ for ``uncovered"), so that $\#U(S) = m - 2\#S$.

\end{definition}

\begin{proposition}\label{prop:a-expr}
For $(g,\Delta) \in \k \times \k^\times$, where $\k \cong \bF_{q^m}$, we have
\[
a(g,\Delta) = \sum_{\substack{S \subset \Z/m \\ \text{disconnected}}} \ \prod_{i \in S} u_i
\prod_{j \in U(S)} w_j .
\]
\end{proposition}

\begin{proof}
Let $M_i := \sigma^i(F)$. By writing out explicitly the diagonal entries of the product $\prod M_i$, we obtain
\[
\Tr(M_0 M_1 \cdots M_{m-1}) = \sum_{s\: \Z/m \to \{1,2\}} \ \prod_{i \in \Z/m}
(M_i)_{s(i),\, s(i+1)} .
\]
Note that $(M_i)_{1,1} = 0$ for all $i$. Thus if $s\: \Z/m \to \{1,2\}$ is such that $s(i) = s(i+1) = 1$ for some~$i$, then the corresponding term is zero. It follows that if the term corresponding to $s$ is nonzero, then $S := s^{-1}(1) \subset \Z/m$ is disconnected. 
The identifications $(M_i)_{1,2} = u_i$, $(M_i)_{2,1} = 1$, and
$(M_i)_{2,2} = w_i$ now give the claimed expression.
\end{proof}

Substituting the definitions of $u_i$ and $w_i$, we obtain:

\begin{corollary}\label{cor:rewrite}
    For the Drinfeld module $\phi^{(g,\Delta)}$ over $\bF_{\p^n}$, we have
\begin{equation}\label{eq:term}
a(g,\Delta) = \sum_{\substack{S \subset \Z/nd \\ \text{disconnected}}}  \prod_{i \in S} (T - \theta^{q^i})
\ \cdot\ (-1)^{\#U(S)}\cdot g^{\sum_{j \in U(S)} q^j}\cdot
\Delta^{-\sum_{i \in S} q^i - \sum_{j \in U(S)} q^j}.
\end{equation}
In particular, for $h \geq 0$, we have
\begin{equation}\label{eq:coeff}
a_h(g,\Delta) = \sum_{\substack{S \subset \Z/nd \text{ disconnected}\\ \#S \geq h}} c_{h}(S) \cdot g^{\sum_{j \in U(S)} q^j} \cdot \Delta^{-\sum_{i \in S}q^i - \sum_{j \in U(S)}q^j},
\end{equation}
where $c_h(S) \in \bF_{q^{nd}}$ is the coefficient of $T^h$ in $(-1)^{\#U(S)}\prod_{i \in S}(T-\theta^{q^i})$.
\end{corollary}

\begin{proof}
    Only the condition $\#S \geq h$ in \eqref{eq:coeff} needs to be justified. For this, simply note that the summand corresponding to $S$ in \eqref{eq:term} has degree~$\#S$, so that $\#S < h \implies c_h(S) = 0$.
\end{proof}

\begin{remark}
    Since a disconnected subset of $\Z/m$ has cardinality at most $\lfloor m/2 \rfloor$, we recover the well-known fact that $\deg a_\phi \leq m/2$ for $\phi$ over $\bF_{q^m}$ (the Riemann hypothesis for rank~2 Drinfeld modules).
\end{remark}

\subsubsection*{Digit sums}

The final remaining ingredient for the proof of Theorem~\ref{thm:SRB_upgrade} is a digit sum lemma.

For an integer $E \geq 0$, write its base $q$ expansion as $E = \sum_{i \geq 0}E_iq^i$, with $0 \leq E_i < q$ for all $i$. Define the \emph{$q$-digit sum function} $s_q$ via $s_q(E) := \sum_{i \geq 0} E_i$. Note that $s_q(E + E') \leq s_q(E) + s_q(E')$, with equality if and only if there are no carries when adding $E$ and $E'$ in base $q$.

\begin{lemma}\label{lem:digitsum}
    Suppose $E > 0$ and $q^m - 1 \mid E$. Then $s_q(E) \geq m(q-1)$.
\end{lemma}

\begin{proof}
    Suppose first that $0 < E < q^m$. Then by hypothesis $E = q^m-1$, and so $s_q(E) = m(q-1)$. If $E \geq q^m$, we may write $E = q^mF + G$ with $F \geq 1$ and $0 \leq G < q^m$. Then $E' := F + G > 0$ is divisible by $q^m-1$ and $s_q(E) = s_q(q^mF) + s_q(G) = s_q(F) + s_q(G) \geq s_q(E')$, so we are done by induction.
\end{proof}

\begin{proof}[Proof of Theorem~\ref{thm:SRB_upgrade}]
    Let $\mathbf{v}$ be as in the statement of the theorem. By Corollary~\ref{cor:mass} and Lemma~\ref{lem:motive-basis}.5, we may write
    \begin{equation}\label{eq:mu(v,t)}
     \mu(\mathbf{v},t) = \sum_{(g,\Delta) \in \k \times \k^\times} (-1)^{ndt}\Delta^{-t\frac{q^{nd}-1}{q-1}} \prod_{i=0}^N a_i(g,\Delta)^{v_i}.
    \end{equation}
    By Corollary~\ref{cor:rewrite}, we have
    \[
    a_i(g,\Delta) = \sum_{\substack{S \subset \Z/nd \text{ disconnected} \\ \#S \geq i}} c_{i}(S,\Delta) \, g^{E(S)},
    \]
    for some $c_{i}(S,\Delta) \in \k$, where we let $E(S) := \sum_{j \in U(S)} q^j$. Note that $\# U(S) = nd - 2\#S \leq nd - 2i$ for all~$S$ appearing in the sum, so in particular $s_q(E(S)) \leq nd - 2i$. 
    
    Expanding the product in \eqref{eq:mu(v,t)}, we obtain
    \[
    \mu(\mathbf{v},t) = \sum_{\Delta \in \k^\times} \sum_{\mathbf{S}} c(\mathbf{S},\Delta,t) \sum_{g \in \k}g^{E(\mathbf{S})}, 
    \]
    where $\mathbf{S}$ ranges over tuples $(\mathbf{S}_i)_{i=0}^{N}$ of tuples $\mathbf{S}_i = (S_{i,j})_{j=1}^{v_i}$ of disconnected subsets satisfying $\#S_{i,j} \geq i$ for all~$j$. Here $E(\mathbf{S}) := \sum_{i=0}^{N} \sum_{j =1}^{v_i} E(S_{i,j})$, and the $c(\mathbf{S},\Delta,t)$ are constants not depending on~$g$.
    
    
    For such $E = E(\mathbf{S})$, we have
    \[
    s_q(E) \leq \sum_{i=0}^N \sum_{j=1}^{v_i} s_q(E(S_{i,j})) \leq \sum_{i=0}^N v_i(nd-2i).
    \]
    By assumption, we have $2 \sum_{i=0}^N iv_i > nd \left( \sum_{i=0}^N v_i - q + 1 \right)$, and hence $s_q(E) < nd(q-1)$. In particular, by Lemma~\ref{lem:digitsum}, we find that either $E = 0$ or $q^{nd}-1$ does not divide $E$. In either case, Lemma~\ref{lem:orth} gives $\sum_{g \in \k} g^E = 0$,
    and hence $ \mu(\mathbf{v},t) = 0$.
\end{proof}

By a similar argument, we obtain a useful formula for $W_{k,t}$. To ease notation, we make the following definition.

\begin{definition}
    Let $k$ and~$m$ be positive integers. Let $\mathbf{S} = (S_1,\ldots,S_k)$ be a $k$-tuple of disconnected subsets of~$\mathbb Z/m$. Define
    \[
        E(\mathbf{S}) := \sum_{\alpha = 1}^k \sum_{j \in U(S_\alpha)} q^j, \qquad C(\mathbf{S},t) := t\frac{q^{m}-1}{q-1} + \sum_{\alpha=1}^k \left( \sum_{i \in S_\alpha} q^i + \sum_{j \in U(S_\alpha)} q^j \right).
    \]
    Then $\mathbf{S}$ is called \emph{$t$-admissible} if $E(\mathbf{S}) > 0$ and $E(\mathbf{S}) \equiv C(\mathbf{S},t) \equiv 0 \pmod{q^m-1}$.
\end{definition}

\begin{proposition}\label{prop:exact}
For all $k\ge0$ and $t\in\Z$,
\[
W_{k,t}\;=\;(-1)^{nd(k+t)}\sum_{\mathbf{S}}\ \prod_{\alpha=1}^{k}\prod_{i\in S_\alpha}\bigl(T-\theta^{q^i}\bigr),
\]
the sum being over $t$-admissible $k$-tuples $\mathbf{S}=(S_1,\dots,S_k)$ of disconnected subsets of $\Z/nd$.
\end{proposition}

\begin{proof}
This follows by substituting \eqref{eq:term} into the definition of~$W_{k,t}$, expanding $a^k$, and summing over $(g,\Delta) \in \k \times \k^\times$.
\end{proof}



\begin{remark}\label{rmk:W-for-deg1}
    If $d=1$, then $\theta \in \bF_q$ and hence $T-\theta^{q^i} = \wp$ for all~$i$. In this case, Prop.~\ref{prop:exact} reduces to
    \[
    W_{k,t} = (-1)^{n(k+t)} \sum_{\mathbf{S}} \wp^{\# \mathbf{S}},
    \]
    where $\mathbf{S}$ runs over $t$-admissible $k$-tuples of disconnected subsets of $\Z/n$, and $\#\mathbf{S} := \sum_{\alpha=1}^k \# S_\alpha$.

    In particular, there are polynomials $f_{k,l,n} \in \bF_p[X] \subset A[X]$ such that for each prime~$\p$ of degree~1, $\Tr(\T_{\p}^n \, | \, \S_{k,l}) = f_{k,l,n}(\wp)$.
\end{remark}

\begin{remark}
    If $nd=1$, the only disconnected subset is the empty set. In this case Proposition~\ref{prop:exact} says that $W_{k,t} = 0$ unless $k > 0$ and $k \equiv t \equiv 0 \pmod{q-1}$, in which case $W_{k,t}=1$. This gives another way to prove the closed-form formula for degree one primes \cite[Thm.~3.13]{devries_traces}.
\end{remark}

\begin{remark}\label{rmk:ram-sharp}
    The strong Ramanujan bound is truly a level~1 phenomenon. The eigenvalue form of the usual Ramanujan bound for $A = \bF_q[T]$ is the inequality
    \[
    -v_{\tilde{\infty}}(\lambda) \leq \frac{d(k-2)}{2},
    \]
    where $\lambda$ is a Hecke eigenvalue on $\S_{k,l}(\Gamma)$ and $v_{\tilde{\infty}}$ is any valuation on~$K(\lambda)$ extending~$v_{\infty}$. If $\Gamma = \Gamma_0(\p)$, then there is a Drinfeld modular form $E_{\p} \in \S_{2,1}(\Gamma)$ such that $\T_{\l}E_{\p} = E_{\p}$ for all $\l \neq \p$ (see e.g.~\cite[\textsection 5.3.1]{BVdV-structure}). Thus $E_{\p}$ violates the strong Ramanujan bound, and in fact for any $n \geq 0$ this gives a Hecke eigenform $E_{\p}^{p^n} \in \S_{2p^n,p^n}(\Gamma)$ with $\T_{\l}$-eigenvalue $\l^{p^n-1}$ (here $\l$ is identified with its monic generator). Hence the usual Ramanujan bound is attained in infinitely many weights.
\end{remark}

\begin{remark}
    The trace form of the strong Ramanujan bound is weaker than the eigenvalue form, and may be satisfied trivially because of $p$-fold repetition of eigenvalues. In particular, this paper does not settle Conjecture~\cite[Conj.~5.24]{devries_traces} even when $v=v_\infty$.
    
    This issue rears its head when considering higher levels, as shown in \cite{hattori_constancy} for level $\Gamma_1(\mathfrak{n}\wp^r)$ with $\mathfrak{n} \trianglelefteq A$ nonzero, $\p \nmid \mathfrak{n}$, and $r \geq 2$. Similarly, it follows from the trace formula that traces of Hecke operators vanish identically at full level $\Gamma(\mathfrak{n})$ whenever $\mathfrak{n} \neq A$: if $\phi$ admits a full level $\mathfrak{n}$-structure, then it admits $\#\operatorname{GL}_2(A/\mathfrak{n})$ such structures, and this number is divisible by~$q$. 
    
    At level $\mathrm{GL}_2(\bF_q[T])$ with $q$ odd, however, we expect that Hecke eigenvalues are never repeated in a given weight and type. If this is true, the trace form of the strong Ramanujan bound is equivalent to the eigenvalue form (this is well known). On the other hand, the results of this paper do not imply the strong Ramanujan bound on the repeated Hecke eigenvalues in characteristic~2, of which there are known to be infinitely many \cite[Sec.~3]{devries_traces}. As such, it would be valuable to have a more conceptual proof that does not rely on the trace formula.
\end{remark}

From the point of view of class numbers, the strong Ramanujan bound is a remarkable statement. It gives many class number identities mod~$p$ that the author does not know how to prove directly (i.e., without using the results in this paper) in this generality.

\begin{example}\label{ex:class-number-rels}
    Let $q=5$, $\wp = T$, and $n=3$. We verify Corollary~\ref{cor:fixed-b} when $b=1$. An isogeny class is determined by its Frobenius polynomial $c_a(X) := X^2 - aX + T^3$, where $a = a_0 + a_1T$.
    
    By Yu's classification \cite{yu_isogs}, there is precisely one supersingular isogeny class, given by $a=0$, and $\# \Iso_{T^3}(0,1) = H(A[\sqrt{-T}]) = 1$. The other isogeny classes are ordinary and satisfy $a_0 \neq 0$. In each case, we have
    \[
    \# \Iso_{T^3}(a,1) = H(A[\pi_a]),
    \]
    where $\pi_a$ is a root of $c_a(X)$. Note that $\infty$ ramifies in $K(\pi_a)$ since the discriminant $D_a(T) := a(T)^2 - 4T^3$ has odd degree.
    We now have two cases.
    \begin{itemize}
        \item If $D_a$ is square-free, then $c_a(X)$ defines an elliptic curve $E_a$ over $\bF_5$, and $H(A[\pi_a]) = \# E_a(\bF_5)$. 
        \item If $D_a$ factors as $h_af_a^2$ with $f_a$ monic of degree~1, then $K(\pi_a)$ has genus~0. In this case $H(A[\pi_a]) = 6 - \left( \frac{h_a}{f_a} \right)$.
        This occurs precisely when $a \in \{ \pm (1+2T), \pm (3+4T)\}$. In each of these four cases we have $\{h_a,f_a\} = \{ T \pm 1 \}$ and hence $H(A[\pi_a]) = 7$.
    \end{itemize}

    The following table shows $\#\Iso_{T^3}(a,1)$ for all~$a$.

\begin{table}[ht]
\centering
\begin{tabular}{c | ccccc}
\diagbox{$a_0$}{$a_1$} & 0 & 1 & 2 & 3 & 4 \\
\hline
0      & 1 & 0 & 0 & 0 & 0 \\
1      & 6 & 6 & 7 & 3 & 9 \\
2      & 6 & 7 & 9 & 6 & 3 \\
3      & 6 & 3 & 6 & 9 & 7 \\
4      & 6 & 9 & 3 & 7 & 6 \\
\end{tabular}
\caption{The integers $\# \Iso_{T^3}(a,1)$ with $a = a_0 + a_1T$.}
\label{tab:diagonal_table}
\end{table}

Corollary~\ref{cor:fixed-b} for $b=1$ now gives 18 relations between these numbers mod~5, which can all be verified explicitly. Some of them can be seen easily: for instance, each column sums to~$25 \equiv_5 0$, and hence $\sum_a \#\Iso_{T^3}(a,1)a_1^{v_1} \equiv_5 0$ for all $0 \leq v_1 \leq 4$.
\end{example}

\section{Sharpness}\label{sec:sharpness}

\subsection{Discussion of results}

\subsubsection*{Sharpness of the strong Ramanujan bound}

Knowing that the strong Ramanujan bound holds, a natural question is: when is it sharp? By our description of the trace in terms of disconnected subsets, this becomes a purely combinatorial problem. 

The answer is surprising: sharpness in a given weight and type does not depend upon $\p^n$ itself, but only on the parity of~$nd$.

A role is also played by the integer $\nu_{k,l} := \frac{k+2-2l}{q-1}$. Note that when $q$ is odd, $\nu_{k,l}$ is even (resp.~odd) precisely when $l \equiv_{q-1} (k+2)/2$ (resp.~$l \equiv_{q-1} (k+2)/2 + (q-1)/2$), while if $q$ is even, we simply have $\nu_{k,l} \equiv_2 k$.


\begin{theorem}\label{thm:sharpness-criterion}
     Let $k+2 \equiv_{q-1} 2l$. The strong Ramanujan bound for $\T_{\p}^n$ acting on $\S_{k+2,l}$ is sharp if and only if 
    \[
    \begin{cases}
        2q \mid k-(q-1) \text{ and } \nu_{k,l} \text{ is odd} & \text{if } nd \text{ is odd}; \\
        q \mid k-(q-1) \text{ and } p \nmid N_{k,l} & \text{if } nd \text{ is even},
    \end{cases}
    \]
    where $N_{k,l} = \# \{ 0 \leq c \leq (k-(q-1))/q \ | \ c \equiv_{q+1} \nu_{k,l} - l \}$. In the first case, the top coefficient in the trace is $(-1)^{l-1}$. In the second case, the top coefficient is $N_{k,l}$.
\end{theorem}

\begin{remark}
    It follows from Theorem~\ref{thm:sharpness-criterion} that the strong Ramanujan bound for any fixed $\T_{\p^n}$ and~$l$ is sharp in infinitely many weights.
    
    If $q$ is even, the two conditions coincide, and are both equivalent to the binomial coefficient $\binom{k-R(k)/2}{R(k)/2}$ being odd, where $R(k) = k-(q-1)$ (as should be the case by \cite[Thm.~3.26]{devries_traces}). In this case the weight determines the type since $k+2 \equiv_{q-1} 2l$ has a unique solution, and 
    the density of weights $k$ for which the strong Ramanujan bound is attained is $1/2q$.

    If $q$ is odd, the congruence $k+2 \equiv_{q-1}2l$ ensures that $k$ is even. Then for a fixed~$l$, the density of weights congruent to~$2l$ in which the strong Ramanujan bound is attained is $\delta$, with
    \[
    \delta = \begin{cases}
        1/2q & \text{if } nd \text{ is odd};\\
        (p-1)/pq & \text{if } nd \text{ is even.}
    \end{cases}
    \]
    Moreover, if $q$ is odd, there are zero or two solutions to the equation $k+2 \equiv_{q-1} 2l$, namely $l_1$ and $l_2 = l_1 + (q-1)/2$, and if $nd$ is odd, the strong Ramanujan bound is attained on at most one of the spaces $\S_{k+2,l_1}$ and $\S_{k+2,l_2}$. 
\end{remark}

We can also determine sharpness of the bound on the $W_{k,t}$. In particular, we observe that if $nd$ is odd, it is attained only once.

\begin{proposition}\label{prop:W-sharpness}
    Suppose $nd$ is odd. Then $\deg W_{k,t} < nd(k-(q-1))/2$ unless $k = q-1$ and $t \equiv_{q-1} 0$, in which case $W_{k,t} = 1$.
\end{proposition}

\subsubsection*{The smallest slope}
If $\p$ is a prime of degree~1, we can also bound the $\p$-adic valuations of the traces from below.

Every $\p$-adic slope of $\T_{\p}$ acting on $\S_{k,l}$ also occurs as a $\p$-adic slope of the Atkin operator $\U_{\p}$ acting on $\S_{k,l}(\Gamma_0(\p))$. In the case $\p = (T)$, Bandini--Valentino proved \cite[Prop.~6.2]{band-val_slopes} that all $\U_{T}$-slopes are at least $l-1$ and that slope $l-1$ has multiplicity at most~1, so the same holds for $\T_{T}$. We will determine in which weights slope~$l-1$ occurs, and what the next possible smallest slope is if it does not.

Let $v_{\mathfrak p}$ denote the $\mathfrak p$-adic valuation, normalized so that $v_{\mathfrak p}(\wp) = 1$.

\begin{theorem}[Dual bound]\label{thm:dual-bound}
    Let $\deg \p = 1$ and $n \geq 1$. Fix $l \in \{1,\ldots,q-1\}$ and let $k+2 \equiv_{q-1} 2l$ and $k+2 > 2l$. Then 
    \[
    v_{\p}(\Tr(\T_{\p}^n \, | \, \S_{k+2,l})) \geq n(l-1),
    \]
    and equality holds if and only if $p \nmid \binom{k-(l-1)}{l-1}$.
    In particular, whether equality holds depends only on $k \pmod{q}$.
\end{theorem}

\begin{remark}
    If $q = p$, then equality holds in Theorem~\ref{thm:dual-bound} if and only if $k+2 \not \equiv_q l+m$ for any $1 \leq m \leq l-1$. In particular, the density of weights for which the dual bound is sharp is then $1 - \frac{l-1}{q}$.
\end{remark}

Combined with work of Bandini--Valentino, we obtain more information about the smallest $\p$-adic slope, which is in line with \cite[Conj.~5.27]{devries_traces}.

\begin{corollary}\label{cor:smallest-slope}
    Suppose $\deg \p = 1$. Fix $l \in \{1,\ldots,q-1\}$ and suppose $\S_{k+2,l} \neq 0$. Then the smallest $\p$-adic slope of $\T_{\p}$ on $\S_{k+2,l}$ is $l-1$ if and only if $p \nmid \binom{k-(l-1)}{l-1}$. If slope~$l-1$ occurs, it does so with multiplicity~1. If it does not, then the smallest slope is at least $\alpha := l-1 + (q-1)/2$, and $\alpha$ occurs with multiplicity~0 or~2.
\end{corollary}

\begin{proof}
    Since the ring automorphisms of $A$ act transitively on primes of degree~1, we may assume $\wp = T$. By \cite[Prop.~6.2]{band-val_slopes} we know that in type~$l$, the smallest $\U_T$-slope at level $\Gamma_0(T)$ is $l-1$, and it occurs with multiplicity $\leq 1$. Since each $\T_{T}$-slope occurs as a $\U_T$-slope, with $\U_T$-multiplicity at least the $\T_{T}$-multiplicity (and equal multiplicities if the slope is not $k/2$), we know that the smallest slope of $\T_{T}$ is not repeated $p$~times and can hence be read off from the traces, cf.~\cite[Prop.~4.2]{devries_newton}. Therefore Theorem~\ref{thm:dual-bound} determines when this slope occurs.

    Now suppose slope $l-1$ does not occur. In the notation of \cite[Sec.~6]{band-val_slopes} (but with our normalization), the characteristic polynomial of~$\U_T$ is then
    \[
    X^n + \ell_1 X^{n-1} + \ell_2 X^{n-2} + \ldots + \ell_{n-r}X^r,
    \]
    and $v_T(\ell_i) \geq i(l-1) + i(i-1)(q-1)/2$. Moreover, our assumption implies (by \emph{loc.~cit.}) that $v_T(\ell_1) \geq l-1 + (q-1)$. Considering the Newton polygon, we see that the smallest $\U_T$-slope that can occur is $\alpha$, which happens if and only if $v_T(\ell_2) = 2\alpha$, in which case the slope has multiplicity~2. The multiplicity of $\T_{T}$-slope $\alpha$ is then also 0 or~2, since under our assumptions $\alpha \neq k/2$: indeed, $\alpha = k/2$ would force $k+2=2l+q-1$, in which case $\dim \S_{k+2,l} > 0$ implies $l = 1$. But on $\S_{q+1,1} = \langle h \rangle$ the smallest slope is known to be $0 = l-1$. 
\end{proof}

\begin{remark}
    After proving Corollary~\ref{cor:smallest-slope}, we noted that it is also a consequence of \cite[Prop.~6.2]{band-val_slopes} and \cite[Thm.~3.13]{devries_traces}. However, the proof presented here is different and has some advantages: it actually gives the coefficient of the lowest $\p$-adic term in the trace of $\T_{\p}^n$ for any $n \geq 1$, and it is possible that the method can be adapted to treat primes of higher degree.
\end{remark}

\begin{example}
    If $l=1$, slope~0 always occurs. One does not need Corollary~\ref{cor:smallest-slope} to see this: it also follows from the theory of cusp forms with a simple zero at the cusp.
    
    If $l=2$, our result states that slope $1$ occurs in $\S_{k+2,2}$ if and only if $k \not \equiv_p 1$. This agrees with \cite[Ex.~3.23]{devries_traces}, which states that for any $q > 2$ and $m \geq 0$,
    \[
    \Tr(\T_{T} \, | \, \S_{2(q+1) + m(q-1),2}) \equiv (m+1)T \pmod{T^q}.
    \]
    If $l=q-1$, we find that the dual bound is attained precisely when $k+2 \equiv_q -2, -1$. One famous example of this is $\Delta \in \S_{q^2-1,0}$, which satisfies $\T_{\p}\Delta = \wp^{q-2}\Delta$ for all~$\p$.
\end{example}

\subsection{Combinatorics}

In this subsection, we prove Theorems~\ref{thm:sharpness-criterion} and~\ref{thm:dual-bound}. The strategy will be to first compute the top coefficient of the $W_{k,t}$ via Prop.~\ref{prop:exact}, and then combine these into the top coefficient of the Hecke trace via the trace formula. From the shape of Prop.~\ref{prop:exact}, it will be obvious to the reader that this will require counting certain $t$-admissible $k$-tuples of disconnected subsets. To perform this count, we need to do some bookkeeping.

\begin{definition}
    Let $\mathbf{S} = (S_1,\ldots,S_k)$ be a $k$-tuple of disconnected subsets of $\mathbb Z/m$. Define its \emph{coverage function} $f_{\mathbf{S}}$ as follows:
    \[
    f_{\mathbf{S}} \: \mathbb Z/m \longrightarrow \mathbb Z_{\geq 0}, \qquad i \longmapsto \# \{\alpha \in \{1,\ldots,k\} \ | \ i \in S_{\alpha} \}.
    \]
    If $f\: \mathbb Z/m \to \mathbb Z$ is any function, we define
    \[
    \omega_k(f) := \# \{ \mathbf{S} = (S_1,\ldots,S_k) \ | \ \text{each } S_\alpha \subset \mathbb Z/m \text{ is disconnected and } f_{\mathbf{S}} = f \}.
    \]
\end{definition}


    For a graph-theoretic interpretation of the numbers $\omega_k(f)$, see Remark~\ref{rmk:kneser}.

\begin{remark}
    The fact that each $S_i$ is disconnected implies that $f_{\mathbf{S}}(i) + f_{\mathbf{S}}(i+1) \leq k$ for all~$i\in \mathbb Z/m$. In particular, $\omega_k(f) = 0$ if $f$ does not satisfy this property. Of course, the same holds if $f(i) < 0$ for some~$i$.
    Thus, $\omega_k$ has finite support as a function on $\{ f \: \Z/m \to \Z\}$.
\end{remark}

In the following, we will fix $m \geq 1$ and let $\Phi = \prod_{i \in \Z/m} F_i$, where $F_i := \begin{pmatrix}
    0 & u_i \\ 1 & w_i
\end{pmatrix}$. We write $a := \Tr \Phi$ and $D := \det \Phi$. This is motivated by Section~\ref{sec:srb-proof}, and we will return to this setting eventually, but for now we actually want to consider the $u_i$ and $w_i$ as formal variables over~$\Z$. The reason for this will become clear in the proof of Prop.~\ref{prop:walk-identity}.

\begin{lemma}\label{lem:tr-a^k}
    We have
    \[
    a^k = \sum_f \omega_k(f) x_k^f, \qquad x_k^f := \prod_{j \in \mathbb Z/m} u_j^{f(j)} w_j^{k-f(j)-f(j+1)},
    \]
    where the sum is over all functions $f\: \Z/m \to \Z_{\geq 0}$.
\end{lemma}

\begin{proof}
    The expression from Prop.~\ref{prop:a-expr} holds also over $\Z[u_i,w_i]$. Raising it to the $k$-th power yields the expression
    \[
    a^k = \sum_{\mathbf{S}} \prod_{\alpha = 1}^{k} \prod_{i \in S_\alpha} u_i \prod_{j \in U(S_\alpha)} w_j.
    \]
    In the summand corresponding to $\mathbf{S}$, each $u_i$ occurs with multiplicity $f_{\mathbf{S}}(i)$ and each $w_j$ with multiplicity $k - f_{\mathbf{S}}(j) - f_{\mathbf{S}}(j+1)$. Grouping the tuples with the same coverage function gives the claimed expression.
\end{proof}

\begin{lemma}\label{lem:tr-sym-phi}
    Keeping the previous notation, we have
    \[
    \sum_{0 \leq j \leq k/2} (-1)^j \binom{k-j}{j} a^{k-2j} D^{j} = \sum_f \prod_{j \in \Z/m} \binom{k-f(j+1)}{f(j)} x_k^f,
    \]
    where the sum is over all functions $f\: \Z/m \to \Z_{\geq 0}$ such that $f(i) + f(i+1) \leq k$ for all~$i$.
\end{lemma}

\begin{proof}
    Write $(X-\pi)(X-\bar{\pi}) = X^2 - aX + D$ and let $h_k(x,y) = \sum_{i=0}^k x^iy^{k-i}$. Then by \cite[Lem.~3.3]{devries_traces}, the left-hand side equals $h_k(\pi,\bar{\pi})$. On the other hand, $h_k(\pi,\bar{\pi}) = \Tr(\Sym^k\Phi)$; this will give the right-hand side, as follows. We have $\Sym^k \Phi = \prod_{i=0}^{nd-1}\Sym^k F_i$. Denote by $X, Y$ the standard basis of $M := \Z^2$, so $F_i$ sends $X \mapsto Y$ and $Y \mapsto u_iX + w_iY$. In the basis $\{X^jY^{k-j} \ | \ 0 \leq j \leq k \}$ of $\Sym^k M$, this gives
    \[
    \Sym^k F_i (X^jY^{k-j}) = Y^j (u_iX+w_iY)^{k-j} = \sum_{j'=0}^k \binom{k-j}{j'}u_i^{j'}w_i^{k-j-j'}X^{j'}Y^{k-j'}.
    \]
    Just like in the proof of Prop.~\ref{prop:a-expr}, we obtain
    \[
    \Tr (\Sym^k \Phi) = \Tr \bigg{(} \prod_{i \in \Z/m} \Sym^k F_i\bigg{)} = \sum_f \prod_{j \in \Z/m} \binom{k-f(j+1)}{f(j)} x_k^f,
    \]
    the sum being over all functions $f\: \Z/m \to \{0,1,\ldots,k\}$. If $f(j) + f(j+1) > k$, then the binomial coefficient will be zero, which completes the proof.
\end{proof}

We are now ready to prove a key identity. For a function $f$ and an integer $j$, $f+j$ will denote the function sending $i \mapsto f(i) + j$.

\begin{proposition}\label{prop:walk-identity}
    Let $k \geq 1$ and let $f\: \Z/m \to \Z_{\geq 0}$ satisfy $f(j) + f(j+1) \leq k$ for all~$j$. Then
    \[
    \sum_{j \geq 0} (-1)^{(m+1)j} \binom{k-j}{j} \omega_{k-2j}(f-j) = \prod_{j \in \Z/m} \binom{k - f(j+1)}{f(j)}.
    \]
\end{proposition}

\begin{proof}
    We have $\det \Phi = \prod \det F_i = (-1)^{m} \prod u_i$. Together with Lemma~\ref{lem:tr-a^k}, this gives
    \[
    a^{k-2j}D^{j} = (-1)^{mj} \sum_f \omega_{k-2j}(f) x_k^{f+j}.
    \]
    Plugging this into Lemma~\ref{lem:tr-sym-phi} gives
    \[
    \sum_f \sum_{0 \leq j \leq k/2} (-1)^{(m+1)j} \binom{k-j}{j} \omega_{k-2j}(f)x_k^{f+j} = \sum_f \prod_{j \in \Z/m} \binom{k - f(j+1)}{f(j)} x_k^f,
    \]
    both sums being over functions $f\: \Z/m \to \Z_{\geq 0}$ such that $f(i) + f(i+1) \leq k$.
    We can view this as an identity in $\Z[u_i,w_i]$, in which case the map $f \mapsto x_k^f$ is injective; comparing coefficients completes the proof.
\end{proof}

We now turn to the top coefficient of $W_{k,t}$. For the rest of this section, we fix a prime $\p$ of degree~$d$ and an integer $n \geq 1$. For ease of notation, write $R(k) := k-(q-1)$. We will describe $[T^{ndR(k)/2}]W_{k,t}$ in terms of coverage functions, which leads to the following notion.

\begin{definition}
    A function $f\: \Z/m \to \Z$ is called \emph{$r$-balanced} if $f(i) + f(i+1) = r$ for all $i \in \Z/m$.
\end{definition}

\begin{proposition}\label{prop:W-top-coeff}
    Suppose $ndR(k)$ is even. Then we have
    \[
    [T^{ndR(k)/2}]W_{k,t} = (-1)^{nd(k+t)} \sum_{R(k)\text{-balanced }f\in \mathcal{A}_t } \omega_k(f),
    \]
    where 
    \[
    \mathcal{A}_t = \left\{ f \: \Z/nd \to \Z \ \Bigg{|} \ \sum_{j=0}^{nd-1} f(j)q^j \equiv_{q^{nd}-1} -t \frac{q^{nd}-1}{q-1} \right\}.
    \]
\end{proposition}

\begin{proof}
    We first rewrite Prop.~\ref{prop:exact} in terms of coverage functions. For $f\: \Z/nd \to \Z$, write $x_j := k - f(j) - f(j+1)$. If $f = f_{\mathbf{S}}$, one checks that $x_j = \# \{ 1 \leq \alpha \leq k \ |\ j \in U(S_\alpha)\}$. Thus, we define
    \[
    E(f) := \sum_{j=0}^{nd-1} x_j q^j, \qquad C(f,t) := t \frac{q^{nd}-1}{q-1} + \sum_{j=0}^{nd-1} (f(j)+x_j)q^j,
    \]
    and call $f$ $t$-admissible if $E(f) > 0$ and $E(f) \equiv_{q^{nd}-1} C(f,t) \equiv_{q^{nd}-1} 0$. Note that $E(\mathbf{S}) = E(f_{\mathbf{S}})$ and $C(\mathbf{S},t) = C(f_{\mathbf{S}},t)$. Thus Prop.~\ref{prop:exact} becomes
    \[
    W_{k,t} = (-1)^{nd(k+t)} \sum_{f t\text{-admissible}} \omega_k(f) \prod_{i \in \Z/nd} (T - \theta^{q^i})^{f(i)}.
    \]
    So it suffices to show that the $t$-admissible $f$ with $\sum f(i) = ndR(k)/2$ are precisely the $R(k)$-balanced $f \in \mathcal A_t$.

    If $\sum f(i) = ndR(k)/2$, then $\sum x_j = ndk - 2\sum f(i) = nd(q-1)$. On the other hand,
    \[
    \sum x_j \geq s_q(E(f)) \geq nd(q-1),
    \]
    the second inequality being Lemma~\ref{lem:digitsum}. Thus the first inequality is an equality. But $s_q(\sum x_i q^i) = \sum x_i$ if and only if $x_i \leq q-1$ for all~$i$, and since the sum equals $nd(q-1)$ we must have $x_i = q-1$ for all~$i$. Hence $f(j) + f(j+1) = k - (q-1)$ for all~$j$, i.e., $f$ is $R(k)$-balanced. If $f$ is moreover $t$-admissible, then clearly $f \in \mathcal A_t$. 

    Conversely, if $f$ is $R(k)$-balanced then $x_j = q-1$ for all~$j$, so $E(f) > 0$ and $q^{nd}-1 \mid E(f)$. If moreover $f \in \mathcal A_t$ then also $q^{nd}-1 \mid C(f,t)$, so $f$ is $t$-admissible, as we wanted. 
\end{proof}

The following lemma about $r$-balanced maps explains where the parity of $nd$ comes into play.

\begin{lemma}\label{lem:balanced}
    Let $f\: \Z/m \to \Z_{\geq 0}$ be $r$-balanced for some $r \geq 0$. 
    \begin{enumerate}
        \item If $m$ is odd, then $r$ is even and $f(i) = r/2$ for all~$i$.
        \item If $m$ is even, then $f = f^{(c)}$ for some $0 \leq c \leq r$, where we define
        \[
        f^{(c)}(i) = \begin{cases}
            c & \text{if } i \text{ is even;} \\
            r-c & \text{if } i \text{ is odd.}
        \end{cases}
        \]
    \end{enumerate}
\end{lemma}

\begin{proof}
    This follows from the observation that an $r$-balanced map must be 2-periodic.
\end{proof}

\begin{lemma}\label{lem:B-is-indicator}
    Let $B\: \mathbb{Z}_{\geq 0} \to \bF_p$ be the function $B(m) := \binom{m + q - 1}{q-1}$. Then $B = \mathds{1}_{q \mid m}$ is the indicator function of~$q\mathbb{Z}_{\geq 0}$. 
\end{lemma}

\begin{proof}
    We use Lucas's Theorem. If $q = p^e$, the first $e$ base-$p$ digits of $q-1$ are all $p-1$, so $B(m) = 1$ if and only if the first $e$ base-$p$ digits of $m + q - 1$ are also~$p-1$, and $B(m) = 0$ otherwise. So $B(m) = 1 \iff m+q-1 \equiv_q q-1 \iff q \mid m$. 
\end{proof}

We are now ready to determine sharpness of the strong Ramanujan bound.

\begin{proof}[Proof of Theorem~\ref{thm:sharpness-criterion}]
We use the equation
    \begin{equation}\label{eq:trace-and-W}
        \Tr(\T_{\p}^n \, | \, \S_{k+2,l}) = \sum_{0 \leq j \leq k/2} (-1)^j \binom{k-j}{j} W_{k-2j,j+l-k-1} \wp^{nj},
    \end{equation}
    which is just the Hecke trace formula from the introduction: the term $j = k/2$, if it occurs, is zero by \cite[Lemma~3.4]{devries_traces}. This gives
    \[
    [T^{ndR(k)/2}]\Tr(\T_{\p}^n \, | \, \S_{k+2,l}) = \sum_{0 \leq j \leq k/2} (-1)^j \binom{k-j}{j} [T^{ndR(k-2j)/2}]W_{k-2j,j+l-k-1}.
    \]
    Now plug in the expression from Prop.~\ref{prop:W-top-coeff}. Note that the map $f \mapsto f + i$ defines a bijection
    \[
    \{ (r-2i)\text{-balanced maps } \Z/m \to \Z_{\geq 0} \} \longrightarrow \{ r\text{-balanced maps } \Z/m \to \Z_{\geq i} \}.
    \]
    Moreover, $f \in \mathcal A_t$ if and only if $f + i \in \mathcal A_{t-i}$. Thus we obtain
    \begin{align*}
    [T^{ndR(k)/2}]\Tr(\T_{\p}^n \, | \, \S_{k+2,l}) &= (-1)^{nd(l-1)} \sum_{R(k)\text{-balanced }f \in \mathcal A_{1-l}} \sum_{j \geq 0} (-1)^{j(nd+1)} \binom{k-j}{j} \omega_{k-2j}(f - j) \\
    &= (-1)^{nd(l-1)} \sum_{R(k)\text{-balanced }f \in \mathcal A_{1-l}} \prod_{j \in \Z/nd} \binom{k-f(j+1)}{f(j)},
    \end{align*}
    the last equality being Prop.~\ref{prop:walk-identity}. If $f$ is $R(k)$-balanced then $k - f(j+1) = f(j) + q - 1$, so the binomial coefficient equals $\binom{f(j) + q - 1}{q-1} = B(f(j))$.

    To conclude, first suppose $nd$ is odd. If $R(k)$ is odd, the bound $ndR(k)/2$ is not an integer and hence not attained; moreover, $2q \nmid R(k)$, so the theorem holds. Hence we may assume that $R(k)$ is even. The constant function $R(k)/2$ lies in $\mathcal{A}_{1-l}$ if and only if $R(k)/2 \equiv_{q-1} l-1$. This is equivalent to $\nu_{k,l}$ being odd,
    so by Lemma~\ref{lem:balanced} we obtain
    \[
    [T^{ndR(k)/2}]\Tr(\T_{\p}^n \, | \, \S_{k+2,l}) = \begin{cases}
        (-1)^{l-1} B(R(k)/2) & \text{if } \nu_{k,l} \equiv_2 1; \\
        0 & \text{otherwise.}
    \end{cases}
    \]
    By Lemma~\ref{lem:B-is-indicator}, we get the claimed criterion.
    
    Next, suppose $nd$ is even and let $0 \leq c \leq R(k)$. Then $f^{(c)}$ lies in $\mathcal A_{1-l}$ if and only if 
    \[
    qR(k) - (q-1)c \equiv_{q^2-1} (q+1)(l-1),
    \]
    which is equivalent to $c \equiv_{q+1} l - \nu_{k,l}$. 
    So in this case, again by Lemma~\ref{lem:balanced},
    \[
    [T^{ndR(k)/2}]\Tr(\T_{\p}^n \, | \, \S_{k+2,l}) = \sum_{c \in C} (B(c)B(R(k)-c))^{nd/2},
    \]
    where $C = \{ 0 \leq c \leq R(k) \ | \ c \equiv_{q+1} l - \nu_{k,l} \}$. By Lemma~\ref{lem:B-is-indicator}, $B(c) = B(R(k)-c) = 1$ if and only if $q$ divides both $c$ and~$R(k)$. Therefore the top coefficient equals $\# \{ 0 \leq c' \leq R(k)/q \ | \ qc' \equiv_{q+1} l - \nu_{k,l}\}$, that is, $N_{k,l}$.
\end{proof}

\begin{proof}[Proof of Prop.~\ref{prop:W-sharpness}]
    We use the expression of Prop.~\ref{prop:W-top-coeff}. Since $nd$ is odd, the bound can only be attained if $R(k)$ is even. Write $k = q-1+2s$ with $s \geq 0$. We observed that the only $R(k)$-balanced map is then the constant function $s$, and this lies in $\mathcal{A}_t$ if and only if $s+t \equiv_{q-1} 0$. In this case, the top coefficient of $W_{k,t}$ is $\omega_k(s)$.
    
    If $s=0$, clearly $\omega_k(s) = 1$, so $W_{q-1,0} = 1$ attains the bound. It remains to show that $\omega_k(s)$ is divisible by~$p$ for all~$s \geq 1$. Prop.~\ref{prop:walk-identity} gives
    \begin{align*}
    \omega_k(s)  &= \binom{k-s}{s}^{nd} - \sum_{j=1}^s \binom{k-j}{j} \omega_{k-2j}(s-j) \\
    &= B(s)^{nd} - B(s) - \sum_{j=1}^{s-1} \binom{k-j}{j} \omega_{k-2j}(s-j).
    \end{align*}
    Since $B(s)$ is either 0 or 1, this always vanishes mod~$p$ by induction on~$s$.
\end{proof}

We conclude this section with a proof of the dual bound. This will use the following generalization of Lemma~\ref{lem:digitsum}.

\begin{lemma}\label{lem:digitsum_v2}
    Let $s \in \{1,\ldots,q-1\}$. If $E > 0$ and $E \equiv s \frac{q^m-1}{q-1} \pmod{q^m-1}$, then $s_q(E) \geq ms$. In particular, if $E = \sum_{i \geq 0} b_i q^i$ with $b_i \geq 0$, let $c_j := \displaystyle{\sum_{i \equiv_m \,j}} b_i$. Then $\sum b_i \geq ms$, and equality holds if and only if $c_j = s$ for all~$j \in \Z/m$.
\end{lemma}

\begin{proof}
    The proof is similar to that of Lemma~\ref{lem:digitsum}, so we only give a sketch. Clearly $\sum b_i \geq s_q(E)$, with equality if and only if $b_i \leq q-1$ for all~$i$. Replacing $E$ by the sum of its base-$q$ blocks of length~$m$ can only decrease the digit sums and eventually ends up in $(0,q^m-1]$. In that case $E = s\frac{q^m-1}{q-1}$ with digit sum $ms$. Equality holds if and only if no carries occurred in this procedure, which is equivalent to $c_j = s$ for all~$j$.
\end{proof}

\begin{proposition}\label{prop:dual-bound-W}
    Suppose $\deg \p = 1$. Fix $t \in \Z$ and $k \geq 0$. Let $s \in \{0,\ldots,q-2\}$ denote the representative of $-t \pmod{q-1}$. Then $v_{\p}(W_{k,t}) \geq ns$, and equality holds if and only if $k > 2s$, $k \equiv_{q-1} 2s$, and $p \nmid \omega_k(f)$ for the constant function $f = s\: \Z/n \to \Z$.
\end{proposition}

\begin{proof}
    Since $W_{k,t} = 0$ unless $k+2t \equiv_{q-1} 0$, we may assume $k \equiv_{q-1} 2s$. First assume $s > 0$.
    
    By Remark~\ref{rmk:W-for-deg1}, it suffices to show that if $\mathbf{S}$ is a $t$-admissible tuple of disconnected subsets of $\mathbb Z/n$, then $\# \mathbf{S} \geq ns$. Let $\mathbf{S}$ be such a tuple. By $t$-admissibility, we have
    \[
    \sum_{\alpha = 1}^k \sum_{i \in S_\alpha} q^i \equiv -t \frac{q^n-1}{q-1} \pmod{q^n-1} \equiv s \frac{q^n-1}{q-1} \pmod{q^n-1}.
    \]
    Thus, writing $B := \sum_{\alpha = 1}^k \sum_{i \in S_{\alpha}} q^i = \sum_{i = 0}^{n-1} f_{\mathbf{S}}(i)q^i$, Lemma~\ref{lem:digitsum_v2} gives $s_q(B) \geq ns$. But then also $\# \mathbf{S}= \sum_{i=0}^{n-1} f_{\mathbf{S}}(i) \geq s_q(B) \geq ns$, as desired. Lemma~\ref{lem:digitsum_v2} also says that we have equality if and only if $f_{\mathbf{S}}(i) = s$ for all~$i$, i.e., $f_{\mathbf{S}}$ is the constant function~$s$, and one checks that this function is $t$-admissible whenever $k > 2s$ and $k \equiv_{q-1} 2s$.
    In this case,
    \[
    [\wp^{ns}] W_{k,t} = (-1)^{n(k+t)}\omega_k(s).
    \]
    Finally, if $s=0$ the inequality is trivial, and the equality is easily seen to hold for each positive multiple~$k$ of~$q-1$ by Remark~\ref{rmk:W-for-deg1}. This agrees with the lemma because $\omega_k(0)=1$.
\end{proof}

\begin{proof}[Proof of Theorem~\ref{thm:dual-bound}]
    We now deduce the statement about the traces from Equation~\eqref{eq:trace-and-W}.
    We may assume $k+2 \equiv_{q-1} 2l$. Denoting by $[c]_{q-1}$ the representative in $\{0,\ldots,q-2\}$ of~$c \in \mathbb Z$, the $j$-term has valuation at least $nj + n[l-1-j]_{q-1} \geq n[l-1]_{q-1} = n(l-1)$; here the inequality is an equality for $0 \leq j \leq l-1$. This proves the claimed inequality. 
    The question of equality remains. Since we assume $k > 2(l-1)$, the above yields
    \[
    [\wp^{n(l-1)}] \Tr(\T_{\p}^n \, | \, \S_{k+2,l}) = (-1)^{n(l-1)}\sum_{0 \leq j \leq l-1} (-1)^{(n+1)j} \binom{k-j}{j} \omega_{k-2j}(l-1-j).
    \]
    By Prop.~\ref{prop:walk-identity}, this simplifies to
    \[
    (-1)^{n(l-1)} \binom{k-(l-1)}{l-1}^n.
    \]
    By Lucas's Theorem, this vanishes if and only if $\binom{a_0}{l-1} \equiv_p 0$, where $a_0$ denotes the $q^0$-digit of $k-(l-1)$; in particular, it depends only on $k \pmod{q}$.
\end{proof}

\begin{remark}\label{rmk:kneser}
    The notions introduced in this section have graph-theoretic interpretations. For any $k \geq 1$, denote by $D_k$ the \emph{disjointness graph} on the set $\{1,\ldots,k\}$: this is the graph with vertices given by subsets of $\{1,\ldots,k\}$ and edges $(V_i,V_j)$ if and only if $V_i \cap V_j = \emptyset$. In particular, the only vertex with a loop is the empty set.

    A \emph{closed walk of length~$m$} in a graph $G$ is a map $V\: \Z/m \to G$ such that $(V(i),V(i+1))$ is an edge for all~$i$. Now observe that a closed walk of length~$m$ in $D_1$ is the same as (the coverage function of) a disconnected subset $S \subset \Z/m$: the function $V\: \Z/m \to \{ \emptyset, \{1\}\} \xrightarrow{\sim} \{0,1\}$ satisfies the characterizing property $V(i) = 1 \implies V(i+1) = 0$.

    More generally, $V\: \Z/m \to D_k$ is a closed walk if and only if $V(i) \cap V(i+1) = \emptyset$ for all~$i$, so the $k$-tuple $\mathbf{S} = (S_\alpha)$ with $S_{\alpha} = \{ i \in \Z/m \ | \ \alpha \in V(i)\}$ consists of disconnected subsets. Thus, for arbitrary $f\: \Z/m \to \Z$, we have the interpretation $\omega_k(f) = \# \Omega_k(f)$, where
    \[
    \Omega_k(f) := \{ \text{closed walks }V\text{ of length }m\text{ in }D_k\text{ such that }\#V(i) = f(i) \text{ for all }i\}. 
    \]
    If $f = s$ is constant, then this counts closed walks in the Kneser graph $K(k,s)$, defined as the induced subgraph of $D_k$ with vertices the $s$-element subsets. In particular, if $m$ is odd and $f$ is $r$-balanced, $\omega_k(f)$ counts closed walks in $K(k,r/2)$. In this case statements such as Prop.~\ref{prop:W-sharpness} and Prop.~\ref{prop:walk-identity} acquire graph-theoretic content: the former states that the number of closed walks of any odd length~$m$ in $K(2s+q-1,s)$ is divisible by~$p$ for every $s \geq 1$. This can also be proved using the eigenvalues of the adjacency matrices of Kneser graphs \cite[Thm.~9.4.3]{alg-graph-thy}.
    
    It would be interesting to explore to what extent this connection can be used in the study of Drinfeld modular forms and/or graph theory.
\end{remark}


\section{The mod $p$ distribution of Frobenius polynomials}\label{sec:distributions}

In \cite{gek_frobdist}, some results were obtained about the distribution of Frobenius polynomials, that is, the distribution of the integers $\# \Iso_{\p^n}(a,b)$ as $a$~and $b$ vary. In particular, in Section~4 of \emph{loc.~cit.}, Gekeler studied how often $\deg a$ is maximal by considering the sets
    \begin{align*}
    H_+(\p) &= \{ [\phi] / \bF_{\p} \ | \ a_{N,\phi} \neq 0\},\\
    H_-(\p) &= \{ [\phi] / \bF_{\p} \ | \ a_{N,\phi} = 0\}.
    \end{align*}
    He then found explicit expressions for the cardinalities of these sets, depending on the parity of~$d$. In particular, he proved that
    \begin{equation}\label{eq:gekeler_counts}
    \# H_-(\p) = \begin{cases}
        (q-1)q^{d-1} & \text{if } d \text{ is odd};\\
        (q-1) \left( \frac{q^d-1}{q+1} + 1\right) & \text{if } d \text{ is even}.
        \end{cases}
    \end{equation}

The methods employed in this paper are closely related to Frobenius distributions; for instance, Theorem~\ref{thm:SRB_upgrade} is a statement about moment vanishing. On the other hand, we never work with the integers $\# \Iso_{\p^n}(a,b)$ in full, but only their mod~$p$ reductions.

In this section we prove some results about the distribution of $\# \Iso_{\p^n}(a,b) \pmod{p}$, and, where possible, also modulo powers of~$p$. One of the main results, Theorem~\ref{thm:top-invisibility}, can be seen as a variant of \eqref{eq:gekeler_counts} that gives a lower bound on the $q$-adic valuation of $\# H_{-}(\p^n)$ for any $n \geq 1$, while at the same time allowing for more coefficients of the Frobenius polynomial than just $a_N$ to be fixed.

In the following, we will identify a polynomial $a = \sum a_iT^i$ of degree $\leq N$ with the vector $(a_0,a_1,\ldots,a_N) \in \bF_{q}^{N+1}$. We will also use the notation introduced in Section~\ref{sec:notation}.

\begin{lemma}\label{lem:determine}
    Let $f\: \bF_{q}^{N+1} \times \bF_q^\times \to \bF_q$ be a function. Then $f$ is determined by the values $\sum_{a,b}f(a,b)a_0^{v_0}a_1^{v_1}\cdots a_N^{v_N}b^t$ with $0 \leq v_i \leq q-1$ and $0 \leq t \leq q-2$. 
\end{lemma}

\begin{proof}
    Each function $f\: \bF_q \to \bF_q$ is represented by a polynomial $f(\alpha)$, so $\{\text{functions }\bF_q \to \bF_q\} \cong \bF_q[\alpha]/(\alpha^q-\alpha)$, which has $\bF_q$-basis $\{1,\alpha,\ldots,\alpha^{q-1}\}$. Similarly, $\{\text{functions }\bF_q^\times \to \bF_q\}$ has basis $1,b,\ldots,b^{q-2}$. Hence functions on the product $\bF_{q}^{N+1} \times \bF_q^\times$ have a corresponding product basis $\{a_0^{v_0}\cdots a_N^{v_N}b^t \ | \ 0 \leq v_i \leq q-1, \, 0 \leq t \leq q-2 \}$.

    Next, the bilinear pairing $(f,g) \mapsto \sum_{a,b}f(a,b)g(a,b) \in \bF_q$ is non-degenerate, as can be seen by pairing against indicator functions. Therefore a function is zero if it pairs to zero with all elements in the given basis, which yields our claim.
\end{proof}

The Hecke action on the spaces $\S_{k,l}$ is completely understood for $k \leq q+1$, and can be used to prove the following lemma. 
 
\begin{lemma}[{\cite[Rmk.~4.8]{devries_traces}}]\label{lem:low_moments}
    If $\mathbf{v} \in \Z_{\geq 0}^{N+1}$ is such that $||\mathbf{v}|| = 0$, then for all $t \in \Z$,
    \[
    \mu(\mathbf{v},t) = \begin{cases}
        1 & \text{if } 0 < v_0 \equiv t \equiv 0 \pmod{q-1};\\
        0 & \text{otherwise.}
    \end{cases}
    \]
\end{lemma}

We can now prove our first main result on Frobenius distributions mod~$p$. 

\begin{theorem}\label{thm:bottom-uniformity}
    For all $(\alpha,\beta) \in \bF_q \times \bF_q^\times$, we have
    \[
    \# \{ [\phi] / \bF_{\p^n} \ | \ a_{0,\phi} = \alpha, \ b_\phi = \beta\} \equiv 1 \pmod{p}.
    \]
\end{theorem}

\begin{proof}
    Let $f(\alpha,\beta) := \sum_{a_0 = \alpha} \# \Iso_{\p^n}(a,\beta)$. We aim to show that $f(\alpha,\beta) \equiv_p 1$ for all $\alpha,\beta$. For $0 \leq k \leq q-1$ and $0 \leq t \leq q-2$, we have
    \[
    \sum_{\alpha,\beta} f(\alpha,\beta)\alpha^k\beta^t
    = \sum_{a,b} \# \Iso_{\p^n}(a,b)a_0^kb^t = \mu((k,0,\ldots,0),t) = \delta_{k,q-1}\delta_{t,0},
    \]
    the last equality being Lemma~\ref{lem:low_moments}. For the constant function $\mathds{1}(\alpha,\beta) = 1$ we compute
    \[
    \sum_{\alpha,\beta}\mathds{1}(\alpha,\beta)\alpha^k\beta^t = \sum_{\alpha}\alpha^k \sum_\beta \beta^t = \delta_{k,q-1}\delta_{t,0},
    \]
    and hence $f = \mathds{1}$ by Lemma~\ref{lem:determine}.
\end{proof}

\begin{corollary}
    Let $\p^n$ and $(\alpha,\beta) \in \bF_q \times \bF_q^\times$ be arbitrary. Then there exists a rank~2 Drinfeld module $\psi$ over $\bF_{\p^n}$ with $a_{0,\psi} = \alpha$ and $b_\psi = \beta$. Let $\pi$ be a root of $X^2 - a_\psi X + b_\psi\wp^n$ and suppose\footnote{This condition says we are not in (the exceptional) case 4 of Yu's classification \cite[Prop.~5]{yu_isogs}.} $[K(\pi) : K] = 2$. Then the Hurwitz class number $H(A[\pi])$ is not divisible by~$p$. In particular, $H(\mathcal O_{K(\pi)})$ is not divisible by~$p$, where $\mathcal O_{K(\pi)}$ denotes the maximal order in~$K(\pi)$.
\end{corollary}

\begin{proof}
    The fact that $\psi$ exists is immediate from Theorem~\ref{thm:bottom-uniformity}. If $\pi \notin K$, we have $\#\Iso_{\p^n}(a,b) \equiv_p H(A[\pi])$, where $H$ denotes the Hurwitz class number defined in \cite[Sec.~6]{gek_frobdist}. The last statement holds since if $H(A[\pi]) \not \equiv_p 0$, then $H(A[\pi]) \equiv_p H(K(\pi)) c_\pi$ with $c_\pi$ a power of~2, see \cite[Lem.~A.1]{devries_traces}.
\end{proof}

Based on some computations, we are led to the following question:

\begin{question}
    Can Theorem~\ref{thm:bottom-uniformity} also be upgraded to a congruence mod~$q$?
\end{question}

\begin{remark}
    If $p=2$, then by \cite[Appendix~A]{devries_traces}, $\#\Iso_{\p^n}(a,b)$ is odd precisely when $a \in \bF_q$. In particular, in point 1 of the above corollary, one can find $\phi$ with $c_\phi(X) = X^2-\alpha X + \beta \wp^n$. This is not true for odd~$q$: for instance, if $\wp^n = T^2$ and $-\beta$ is a square, then $\Iso_{\p^n}(\alpha,\beta) = \emptyset$ for all~$\alpha \in \bF_q$.
\end{remark}

In the other direction, we will show that many sets of isomorphism classes with some fixed $a_{i,\phi}$ and $b_\phi$ have cardinality divisible by (powers of)~$q$. The mere mod~$p$ vanishing can be deduced from Theorem~\ref{thm:SRB_upgrade}. To upgrade this, we will prove that the high coefficients $a_{i,\phi}$ are polynomial functions of low degree, so that we can invoke the Ax--Katz theorem.

\begin{proposition}\label{prop:coeff-is-pol}
    Fix $\Delta \in \bF_{\p^n}^\times$. Then for each $0 \leq j \leq N$, the map $g \mapsto a_j(g,\Delta) \in \bF_q$ is a polynomial function of $g \in \bF_{\p^n}$ over~$\bF_q$ of degree at most~$nd-2j$ in $nd$ variables.
\end{proposition}

\begin{proof}
    By Equation~\eqref{eq:coeff}, it suffices to show that for each disconnected subset $S \subset \Z/nd$ with $\#S \geq j$, the map $g \mapsto g^{\sum_{i \in U(S)} q^i}$ is a polynomial of degree at most~$nd-2j$. This holds because $g \mapsto g^{q^i}$ is $\bF_q$-linear for all~$i$ and therefore represented by a linear $\bF_{\p^n}$-polynomial. Since $\#U(S) = nd - 2\#S \leq nd-2j$, the total degree is then at most~$nd-2j$. Finally, since $a_j(g,\Delta) \in \bF_q$ for all~$g$, the coefficients of (the reduced representative of) 
    this polynomial actually lie in $\bF_q \subset \bF_{\p^n}$.
\end{proof}

For a subset $S \subseteq \{0,\ldots,N\}$, we define the integer $\mu(S)$ as follows: if $S$ is empty or $S = \{nd/2\}$ (the latter case can only happen if $nd$ is even), define $\mu(S) := nd$. Otherwise, set $S' := S \setminus \{nd/2\}$ and define
    \[
    \mu(S) := \max \left\{ 0, \left\lceil \frac{nd - \sum_{j \in S'}(nd-2j)}{\max_{j \in S'} (nd - 2j)} \right\rceil \right\}.
    \]

Given $\boldsymbol{\alpha} = (\alpha_i)_{i \in S} \in \bF_q^S$ and $\beta \in \bF_q^\times$, define
    \begin{align*}
        D(\boldsymbol{\alpha},\beta) &:= \{ (g,\Delta) \in \bF_{\p^n} \times \bF_{\p^n}^\times \ | \ a_i(g,\Delta) = \alpha_i \ \text{for all } i \in S, \ b(\Delta) = \beta \};\\
        H(\boldsymbol{\alpha},\beta) &:= \{ [\phi] / \bF_{\p^n} \ | \ a_{i,\phi} = \alpha_i \ \text{for all } i \in S, \ b_\phi = \beta \}.
    \end{align*}
Write $d(\boldsymbol{\alpha},\beta)$, resp. $h(\boldsymbol{\alpha},\beta)$, for the cardinalities of these sets. 

\begin{theorem}\label{thm:g-count-divisibility}
    Fix $S \subseteq \{0,\ldots,N\}$. Then for any  $\boldsymbol{\alpha} = (\alpha_i)_{i \in S} \in \bF_q^S$ and $\beta \in \bF_q^\times$, we have
    \[
    d(\boldsymbol{\alpha},\beta) \equiv 0 \pmod{q^{\mu(S)}}.
    \]    
\end{theorem}

\begin{proof}
    We will prove that $d'(\boldsymbol{\alpha},\Delta) \equiv 0 \pmod{q^{\mu(S)}}$ for all~$\Delta \in \bF_{\p^n}^\times$, where \[
    d'(\boldsymbol{\alpha},\Delta) := \# D'(\boldsymbol{\alpha},\Delta) := \#\{ g \in \bF_{\p^n} \ | \ a_i(g,\Delta) = \alpha_i \ \text{for all } i \in S \}.
    \]
    This suffices, since then
    $d(\boldsymbol{\alpha},\beta) = \sum_{b(\Delta) = \beta} d'(\boldsymbol{\alpha},\Delta) \equiv 0 \pmod{q^{\mu(S)}}$.
    
    The case $S = \emptyset$ is trivial, since then $d'(\boldsymbol{\alpha},\Delta) = q^{nd}$ for all~$\Delta$ and $\mu(S) = nd$. The case $S = \{ nd/2 \}$ occurs only if~$nd$ is even. By Prop.~\ref{prop:coeff-is-pol}, the map $g \mapsto a_N(g,\Delta)$ is then a constant $\lambda \in \bF_q$. Thus $d'(\boldsymbol{\alpha},\Delta)$ is either 0 (if $\lambda \neq \alpha_N$) or $q^{nd}$ (if $\lambda = \alpha_N$), so in either case the statement holds.
    
    Next suppose $S' \neq \emptyset$. If $nd/2 \in S$, then by the above argument, we either have $d'(\boldsymbol{\alpha},\Delta) = 0$ or $d'(\boldsymbol{\alpha},\Delta) = d'(\boldsymbol{\alpha}',\Delta)$, where $\boldsymbol{\alpha}' = (\alpha_i)_{i \in S'}$. In the first case we are done, so we may assume $S = S'$. Then
    \[
    d'(\boldsymbol{\alpha},\Delta) = d'(\boldsymbol{\alpha}',\Delta) = \# \{\mathbf{x} \in (\bF_q)^{nd} \ | \ f_j(\mathbf{x}) - \alpha_j = 0 \text{ for all } j \in S' \},
    \]
    where $f_j \in \bF_q[X_1,\ldots,X_{nd}]$ is the polynomial representing $g \mapsto a_j(g,\Delta)$ (see Prop.~\ref{prop:coeff-is-pol}).

    Since $\deg f_j \leq nd-2j$, the congruence now follows immediately from the Ax--Katz theorem \cite{wan_ax-katz}.
\end{proof}

\begin{theorem}\label{thm:top-invisibility}
    Fix $S \subseteq \{0,\ldots,N\}$. Then for any  $\boldsymbol{\alpha} = (\alpha_i)_{i \in S} \in \bF_q^S$ and $\beta \in \bF_q^\times$,
    \[
    h(\boldsymbol{\alpha},\beta) \equiv \begin{cases}
        0 \pmod{q^{\mu(S)}} & \text{if } nd \text{ is odd}; \\
        0 \pmod{q^{\min\{1,\mu(S)\}}} & \text{if } nd \text{ is even.}
    \end{cases}
    \]
\end{theorem}

\begin{proof}
    Suppose $nd$ is odd. By the proof of Lemma~\ref{lem:domain}, the number of pairs $(g,\Delta)$ such that $[\phi^{(g,\Delta)}] = [\phi]$ is $(q^m-1)/(q-1)$ for each Drinfeld module $\phi$ over $\bF_{q^m}$. Hence we obtain
    \[
    d(\boldsymbol{\alpha},\beta) = (1 + q + \ldots + q^{nd-1}) h(\boldsymbol{\alpha},\beta),
    \]
    so in particular $d(\boldsymbol{\alpha},\beta)$ and $h(\boldsymbol{\alpha},\beta)$ have the same $p$-adic valuation. Hence this case follows from Theorem~\ref{thm:g-count-divisibility}.

    Now assume $nd$ is even. Let
    \[
    h_0(\boldsymbol{\alpha},\beta) := \# H_0(\boldsymbol{\alpha},\beta) := \#\{[\phi] \in H(\boldsymbol{\alpha},\beta) \ | \ \# \Aut(\phi) = q^2-1\}.
    \]
    Then we have
    \[
    d(\boldsymbol{\alpha},\beta) = \frac{q^{nd}-1}{q^2-1} h_0(\boldsymbol{\alpha},\beta) + \frac{q^{nd}-1}{q-1} (h(\boldsymbol{\alpha},\beta) - h_0(\boldsymbol{\alpha},\beta)).
    \]
    Since $d(\boldsymbol{\alpha},\beta) \equiv_{q^{\mu(S)}} 0$, after rewriting the right-hand side we obtain
    \[
    (1+q^2 + \ldots + q^{nd-2}) (h(\boldsymbol{\alpha},\beta) + q(h(\boldsymbol{\alpha},\beta) -h_0(\boldsymbol{\alpha},\beta))) \equiv_{q^{\mu(S)}} 0.
    \]
    In particular, if $\mu(S) > 0$ then $h(\boldsymbol{\alpha},\beta)$ is divisible by~$q$, as claimed.
\end{proof}

\begin{remark}
    In the case $n=1$ and $S = \{N\}$, the power of $q$ occurring in Theorem~\ref{thm:top-invisibility} is optimal, as evidenced by Equation~\eqref{eq:gekeler_counts}. 
    
    In general, the theorem is meaningful whenever $\mu(S) > 0$. In the case of the top coefficients $S = \{ N, N-1,\ldots, N-j\}$, we have $\mu(S) > 0$ as long as $j(j+1)<nd$ if $nd$ is even and $(j+1)^2 < nd$ if $nd$ is odd.
\end{remark}

Although we have not carried it out, we expect that the results and methods from this section can be adapted to work for Drinfeld modules of arbitrary rank.

\begingroup
\sloppy
 \printbibliography
\endgroup

\end{document}